\documentclass[11pt]{article}

\usepackage[a4paper, left=3cm, top=3cm, right=3cm, bottom=3cm ]{geometry}
\usepackage{authblk}
\usepackage[onehalfspacing]{setspace}

\usepackage[T1]{fontenc}

\usepackage{hyperref}
\usepackage{amsmath,amssymb}
\usepackage{amsthm}

\usepackage{mathtools}

\usepackage{siunitx}
\usepackage{booktabs}
\usepackage{subcaption}

\usepackage{enumitem}

\usepackage{tikz}
\usetikzlibrary{decorations.pathreplacing}

\usepackage{pgfplots}
\pgfplotsset{compat=1.10}
\usetikzlibrary{intersections}

\usepackage[mode=buildnew]{standalone}

\newcommand{\R}{\mathbb{R}}
\newcommand{\N}{\mathbb{N}}

\newcommand{\al}{\alpha}

\DeclarePairedDelimiter{\set}{\{}{\}}

\newcommand{\tik}{\mathcal{T}}

\def\plus{{\boldsymbol{\texttt{+}}}}

\newcommand{\signal}{x}

\newcommand{\data}{y}

\newcommand{\X}{\mathbb X}
\newcommand{\V}{\mathbb V}
\newcommand{\Y}{\mathbb Y}

\newcommand{\reg}{\mathcal R}

\newcommand{\M}{\mathcal M}

\DeclareMathOperator{\lip}{Lip}

\DeclareMathOperator{\argmin}{arg\,min}

\newcommand{\Ao}{A}
\newcommand{\Bo}{B}
\newcommand{\Ro}{R}
\newcommand{\Qo}{Q}
\newcommand{\To}{T}
\newcommand{\Do}{D}
\newcommand{\Co}{C}
\newcommand{\Lo}{L}

\newcommand{\No}{N}

\newcommand{\Po}{P}

\newcommand{\inner}[1]{\left\langle#1\right\rangle}       \newcommand{\abs}[1]{\left|#1\right|}
\newcommand{\norm}[1]{\left\lVert#1\right\rVert}
\newcommand{\snorm}[1]{\lVert#1\rVert}

\DeclareMathOperator{\id}{Id}

\DeclareMathOperator{\ran}{ran}

\newcommand{\bdist}[2]{\inner{\Ro(#1) -  \Ro(#2), #1 - #2}}

\newtheorem{theorem}{Theorem}[section]
\newtheorem{lemma}[theorem]{Lemma}

\newtheorem{proposition}[theorem]{Proposition}

\theoremstyle{definition}
\newtheorem{definition}[theorem]{Definition}
\newtheorem{example}[theorem]{Example}
\newtheorem{assumption}[theorem]{Assumption}

\newtheorem{remark}[theorem]{Remark}

\numberwithin{equation}{section}
\numberwithin{theorem}{section}
\numberwithin{figure}{section}

\allowdisplaybreaks

\title{Convergence analysis of equilibrium methods for inverse problems}

\date{June 16, 2025}

\author[1]{Daniel Obmann}
\author[2]{Gyeongha Hwang}
\author[1]{Markus Haltmeier\thanks{Corresponding author}}

\affil[1]{Department of Mathematics, University of Innsbruck, Austria\authorcr
E-mail:  \texttt{markus.haltmeier@uibk.ac.at}}
 
\affil[2]{Department of Mathematics, 
Yeungnam University,
South Korea\authorcr
E-mail:  \texttt{ghhwang@yu.ac.kr}
}

\begin{document}

\maketitle

\begin{abstract} 

Solving inverse problems \(\Ao \signal = \data\) is central to a variety of practically important fields such as medical imaging, remote sensing, and non-destructive testing. The most successful and theoretically best-understood method is convex variational regularization, where approximate but stable solutions are defined as minimizers of \(\snorm{\Ao(\cdot) - \data^\delta}^2 / 2 + \alpha \reg(\cdot)\), with \(\reg\) a regularization functional. Recent methods such as deep equilibrium models and plug-and-play approaches, however, go beyond variational regularization. Motivated by these innovations, we introduce implicit non-variational (INV) regularization, where approximate solutions are defined as solutions of \(\Ao^*(\Ao \signal - \data^\delta) + \alpha \Ro(\signal) = 0\) for some regularization operator \(\Ro\). When the regularization operator is the gradient of a functional, INV reduces to classical variational regularization. However, in methods like DEQ and PnP, \(\Ro\) is not a gradient field, and the existing theoretical foundation remains incomplete. To address this, we establish stability and convergence results in this broader setting, including convergence rates and stability estimates measured via a absolute Bregman distance.

\end{abstract}

\section{Introduction} \label{sec:introduction}

In many practically important imaging applications, such as medical imaging, remote sensing, and non-destructive testing, it is not possible to measure the  object of interest directly, but only through indirect measurements. Assuming a linear measurement model, recovering the unknown $\signal \in \X$ requires solving the inverse problem
\begin{equation}\label{eq:ip}
\data^\delta = \Ao \signal + z^\delta \,,
\end{equation}
where $\Ao \colon \X \to \Y$ is a linear operator between Hilbert spaces modeling the forward process, $z^\delta$ represents the data perturbation, and $\data^\delta \in \Y$ denotes the measured noisy data.

In many cases, problems of the form \eqref{eq:ip} are ill-posed, meaning that the operator $\Ao$ cannot be inverted uniquely and stably. To obtain reasonable approximate solutions, one must employ regularization methods, which approximate \eqref{eq:ip} by a family of nearby well-posed problems. The prime example is variational regularization \cite{scherzer2009variational}, where approximate solutions are constructed as minimizers of the generalized Tikhonov functional
\begin{equation} \label{eq:tikhonov}
\tik_{\al} (\cdot, \data)  =  \snorm{\Ao(\cdot) - \data^\delta}^2/2 + \alpha \reg(\cdot) \,,
\end{equation}
with $\reg$ denoting the regularization functional and $\alpha > 0$ a tuning parameter that balances data fidelity and stability.

For quite some time, there has been a growing trend toward solving inverse imaging problems using learned components that are at least partially adapted to available data \cite{arridge2019solving, chen2023imaging, gilton2021deep, li2020nett, jin2017deep, kamilov2023plug, lunz2018adversarial, mccann2017convolutional, riccio2022regularization, romano2017little, ryu2019plug, tan2023provably}. While these methods often achieve superior results compared to classical reconstruction techniques, the corresponding theoretical understanding remains underdeveloped in many cases. In particular, a recent trend involves deep equilibrium (DEQ) methods \cite{gilton2021deep} and plug-and-play (PnP) models \cite{venkatakrishnan2013plug}. These methods define approximate solutions by integrating a network into standard iterative schemes, either using predefined networks (in PnP) or adjusting them in an end-to-end manner based on available data (in DEQ). 

In any case, unless the network is of gradient form (which significantly restricts the class of admissible models), these methods do not fit into the variational framework \eqref{eq:tikhonov}. The aim of this paper is to develop a more general regularization theory that covers the non-gradient case and is therefore applicable to DEQ and PnP methods.

\subsection{Implicit non-variational regularization}

Minimizers of the Tikhonov functional \eqref{eq:tikhonov} satisfy the optimality condition $\Ao^*(\Ao \signal - \data^\delta) + \alpha \nabla \reg(\signal) = 0$, where the regularization functional $\reg$ is assumed to be differentiable with gradient.

The gradient form $\nabla \reg(\signal)$ corresponds to a quite restrictive class. To overcome this, we propose and analyze implicit non-variational (INV) regularization that replaces the additive gradient term $\nabla \reg(\signal)$ with a general operator not necessarily of gradient form. More precisely, we define approximate solutions $\signal_\al^\delta$ as solutions of the INV equation
\begin{equation} \label{eq:INV}
    \To_\al (\signal, \data^\delta ) 
    \coloneqq 
    \Ao^*(\Ao \signal - \data^\delta) + \alpha \Ro(\signal) = 0 \,,
\end{equation}
where $\alpha > 0$ is a tuning parameter, $\Ro \colon \X \to \X$ is a potentially learned regularization operator, and $\data^\delta$ satisfies $\norm{\data^\delta - \data} \leq \delta$ for noise-free measurements $\data \in \Y$.

If $\Ro = \nabla \reg$ is of gradient form, the INV equation~\eqref{eq:INV} characterizes critical points of the generalized Tikhonov functional $\tik_{\al}(\cdot, \data)$ defined in \eqref{eq:tikhonov}. However, in this paper, we are interested in the more general case where $\Ro$ is not necessarily of gradient form.

One main application of such a formulation is in Deep Equilibrium Models (DEQ), where $\Ro$ is implemented as a CNN trained in a supervised manner such that the solutions of \eqref{eq:INV} approximate the provided ground truth data. In these methods, \eqref{eq:INV} is reformulated as a fixed-point equation of the form
\begin{equation}
\label{eq:grad}
\signal = \signal - \beta \left( \Ao^*(\Ao \signal - \data^\delta) + \alpha \Ro(\signal) \right),
\end{equation}
which is then solved via an associated fixed-point iteration. The network parameters are optimized so that the fixed points of \eqref{eq:grad} are close to the ground truth training data. This procedure was proposed in \cite{gilton2021deep}, where the specific form \eqref{eq:grad} is referred to as DE-grad to distinguish it from other DEQ variants and to emphasize that it generalizes the standard gradient method for minimizing \eqref{eq:tikhonov}.
Another relevant scenario arises in the context of PnP and RED-type iterative schemes  \cite{kamilov2017plug,romano2017little,fermanian2023pnp}, where $\Ro$ is a pretrained network used in place of the gradient of a regularizer.

\subsection{Main contributions}

In this paper, we analyze solutions of~\eqref{eq:INV} from the viewpoint of regularization theory. The developed theory can be applied to any iterative scheme whose limit points satisfy \eqref{eq:INV}, not restricted to  \eqref{eq:grad}. We provide stability and convergence results together with convergence rates.

To be more specific, our main contributions are as follows.

\begin{enumerate}[label=(\alph*)]
    \item As a first theoretical result, we show that for fixed $\al > 0$, solutions of~\eqref{eq:INV} depend stably on the data $\data^\delta$ (see Theorem~\ref{thm:stability}). Moreover, we show convergence in the sense that for $\data^\delta \to \data$ with exact data $\data \in \ran(\Ao)$, as $\delta \to 0$ and with a suitable parameter choice $\al = \al(\delta)$, solutions of \eqref{eq:INV} converge to solutions of the limiting problem
    \begin{equation} \label{eq:limiting}
    \Ao \signal = \data \quad \text{ and } \quad \Ro(\signal) \in \ker(\Ao)^\perp
    \end{equation}
    (see Theorem~\ref{thm:convergence}). Note that in the special case where $\Ro = \nabla \reg$, this is the first-order optimality condition of the constrained optimization problem $\argmin_\signal \{ \reg(\signal) \mid \Ao \signal = \data \}$, which defines $\reg$-minimizing solutions in variational regularization~\cite{scherzer2009variational}. In this sense, our results generalize the convergence theory from variational regularization \eqref{eq:tikhonov} to the more general non-variational form \eqref{eq:INV}.
    
   \item 
 Moreover, we derive quantitative estimates (so-called convergence rates) if the solution of \eqref{eq:limiting} satisfies the source condition $\Ro(\signal_\plus) \in \ran(\Ao^*)$. Again, this generalizes the standard source condition $\nabla \reg (\signal_\plus) \in \ran(\Ao^*)$ used to obtain convergence rates in variational regularization. In the ill-posed case, it constitutes an abstract smoothness condition for $\Ro(\signal_\plus)$ that strengthens the characterization $\Ro(\signal) \in \ker(\Ao)^\perp$ of the limiting problem. 
    Specifically, we derive error estimates for $d_\Ro(\signal, \signal_\al^\delta) := \abs{\bdist{\signal}{\signal_\al^\delta}}$ and $\norm{\Ao(\signal_\al^\delta) - \data^\delta}$ (see Theorem~\ref{thm:rates}). In analogy with the case where $\Ro$ is the gradient of a functional, we refer to $d_\Ro(\signal, \signal_\al^\delta)$ as the symmetric Bregman pairing and note that it is not a metric in the classical sense.
    
    \item 
For the special case where the regularization operator has the form $\Ro = \id - \Co$ for some contractive map $\Co$ (which we refer to as a contractive residual regularizer), we strengthen the above results. In particular, we derive uniqueness of the regularized problem \eqref{eq:INV} as well as of the limiting problem \eqref{eq:limiting}. Moreover, we show that in this case, convergence in the symmetric Bregman pairing is equivalent to convergence in norm. We further study recoverability with a given $\Ro$ and provide a priori lower bounds on the Lipschitz constant $\lip(\Co)$ necessary to ensure that a desired set of solutions is recoverable in the limit. Additionally, we provide lower bounds on the reconstruction error in cases where recovery fails. 
\end{enumerate}

To the best of our knowledge, theoretical questions from a regularization point of view such as the convergence of solutions to \eqref{eq:INV} as $\alpha, \delta \to 0$ have not been studied so far. We note, however, the work in \cite{ebner2022plug}, which provides a regularization theory for the fixed points of $\Do_\alpha(\id - \beta \Ao^*(\Ao(\cdot) - \data^\delta))$, where $\Do_\alpha$ is a denoiser. This result is relevant for prox-based variants of DEQ \cite{gilton2021deep} and PnP methods \cite{meinhardt2017learning,sun2019online}.

\subsection{Overview}

We begin the paper with a general theoretical analysis in Section~\ref{sec:general}. In Section~\ref{sec:contraction}, we focus on regularization operators that are contractive residual regularizers, demonstrating that \eqref{eq:INV} indeed defines a regularization method and further analyzing recoverability. The paper concludes with a brief summary and outlook in Section~\ref{sec:conclusion}.

\section{Analysis for general regularizers} \label{sec:general}

Let \(\Ao \colon \X \to \Y\) be a linear and continuous mapping between real Hilbert spaces \(\X\) and \(\Y\), and let \(\Ro \colon \X \to \X\) be an operator used for regularization. In this section, we derive stability, convergence, and quantitative estimates for solutions of~\eqref{eq:INV} under relatively weak assumptions on the operator \(\Ro\), as stated below.

Recall that the operator \(\Ro\) is called monotone if \(\bdist{z}{\signal} \geq 0\) for all \(\signal, z \in \X\). Gradients of convex regularizers are, in particular, monotone. We emphasize that the stability and convergence analysis developed in this section does not require monotonicity of \(\Ro\) and, in particular, goes beyond the classical variational regularization framework.

\subsection{Stability}

We begin our analysis by deriving stability of solutions of equation~\eqref{eq:INV}.   Conditions for their existence will be discussed later.

Recall that $g \colon \X \to \R$ is coercive if $g(\signal) \to \infty$ as  $\norm{\signal} \to \infty$.

\begin{assumption}[Condition for stability and convergence] \label{cond:convergence} \hfill
\begin{enumerate}[label=(A\arabic*), leftmargin=4em, topsep=0em, itemsep=0em]
    \item $\forall z \in \X \colon \signal \mapsto \inner{\Ro(\signal), \signal - z}$ is coercive. \label{cond:coercive}

    \item $\Ro$ is weak-to-weak continuous. \label{cond:grad:weakcont}
\end{enumerate}
\end{assumption}

\begin{remark}[Coercivity condition~\ref{cond:coercive}]
Condition~\ref{cond:coercive} will be referred to as the pointwise directional coercivity of~$\Ro$. If \(\Ro = \nabla \reg\) is a gradient field, then \(\inner{ \Ro(x), x - z } = D\reg[x](x - z)\) is the directional derivative of the functional \(\reg\) at point \(x\) in the direction \(x - z\). It measures how rapidly \(\reg\) increases as one moves from \(z\) to \(x\). If \(\Ro\) is a general (non-gradient) vector field, then \(\inner{ \Ro(x), x - z }\) quantifies the alignment of \(\Ro(x)\) with the vector \(x - z\). A positive value means \(\Ro(x)\) points roughly away from \(z\), while a negative value means it points toward \(z\). In particular, the pointwise directional coercivity condition~\ref{cond:coercive} ensures that \(\Ro\) pushes away from any given \(z\) strongly enough.
\end{remark}

\begin{remark}[Continuity Condition~\ref{cond:grad:weakcont}]
Condition~\ref{cond:grad:weakcont} requires  \((\Ro(\signal_k))_k \to \Ro(\hat{\signal})\) weakly whenever \((\signal_k)_k \to \hat{\signal}\) weakly. As we will see later, this guarantees that \eqref{eq:INV} is stable with respect to perturbations in \(\data^\delta\), giving stability of equilibrium methods. This assumption is, for example, satisfied by bounded, linear \(\Ro\), and hence also by deep neural networks given as compositions of bounded, linear operators and (weakly) continuous activation functions. As such, from a deep learning viewpoint, this assumption is typically satisfied and may be viewed as a technical assumption.
\end{remark}

\begin{example} \label{eq:basic}
If $\Ro(\signal) \in \partial \reg(\signal)$ is a selection of a subgradient of a convex and subdifferentiable function~$\reg$, then $\inner{\Ro(\signal), \signal - z} \geq \reg(\signal) - \reg(z)$, and thus pointwise directional coercivity of $\partial \reg$ follows from the coercivity of~$\reg$. Therefore, Assumption~\ref{cond:convergence} is satisfied for $\Ro = \partial \reg$ whenever $\reg$ is coercive and has a weak-to-weak continuous gradient. However, as we show later, Assumption~\ref{cond:convergence} is also satisfied for large classes of operators that cannot be written as gradients. Simple examples include $\Ro = \id - \Co$, where $\Co$ is a linear, non-symmetric mapping with $\norm{\Co} < 1$, or a single-layer neural network $\Co = \Lo_2 \circ \sigma \circ \Lo_1$ for linear  $\Lo_1, \Lo_2$ with $\norm{\Lo_2} \cdot \norm{\Lo_2} < 1$ and  non-expansive $\sigma$; see   Example~\ref{eq:simple}.
\end{example}

In the following, we frequently use the convexity of the data fidelity term $\mathcal{F}(\cdot) = \snorm{\Ao (\cdot) - \data^\delta}^2 / 2$, which gives the inequality
\[
2 \inner{\Ao^*(\Ao \signal - \data^\delta), z - \signal} \leq \snorm{\Ao z - \data^\delta}^2 - \snorm{\Ao \signal - \data^\delta}^2.
\]
In fact, this inequality follows from $\nabla f(x) = \Ao^*(\Ao x - \data^\delta)$ and rearrangement of the first-order characterization of convexity \( \mathcal{F}(z) \geq \mathcal{F}(x) + \inner{\nabla \mathcal{F}(x), z - x} \).

\begin{theorem}[Stability] \label{thm:stability}
Let $\al > 0$ and $(\data_k)_{k\in \N} \in \Y^\N$ with  $\data_k\to \data^\delta$. Then, any sequence $(\signal_k)_{k\in \N}$ satisfying $\To_\al(\signal_k,  \data_k) = 0$ (with $\To_\al$ defined by \eqref{eq:INV}) has a weakly convergent subsequence. Moreover, any weak cluster point of $(\signal_k)_{k\in \N}$  is a solution of $\To_\al(\signal,  \data^\delta) = 0$.
\end{theorem}

\begin{proof}
We begin by showing that the sequence $(\signal_k)_k$ is bounded. By definition of $\signal_k$ and the convexity of the data-fidelity term, for any $z \in \X$,
\begin{equation*}
    \begin{split}
    0 &= \inner{\Ao^* (\Ao \signal_k - \data_k) + \al \Ro(\signal_k), z - \signal_k} \\
    & \leq \frac{1}{2} \norm{\Ao z - \data_k}^2 - \frac{1}{2} \norm{\Ao \signal_k - \data_k}^2 + \al \inner{\Ro(\signal_k), z - \signal_k} \,.
    \end{split}
\end{equation*}
Hence, $\al \inner{\Ro(\signal_k), \signal_k - z} \leq  \norm{\Ao z - \data_k}^2/2$
where $\norm{\Ao z - \data_k}$ is  bounded.  By Condition~\ref{cond:convergence}, $(\signal_k)_{k\in \N}$ is  bounded and by reflexivity of $\X$ it has a weakly convergent subsequence.
If  $\signal_\plus$ is  a weak cluster point of $(\signal_k)_{k\in \N}$,  then taking the limit in equation~\eqref{eq:INV} and using the weak continuity of $\To_\al$ shows  $\To_\al(\signal_\plus,  \data^\delta) = 0$.
\end{proof}

\subsection{Convergence}

The next goal is to show convergence of the regularized solutions for $\delta \to 0$. 

\begin{theorem}[Convergence] \label{thm:convergence}
Let $\data \in \ran(\Ao)$, $(\data_k)_{k\in \N} \in \Y^\N$ satisfy   $\norm{\data_k - \data} \leq \delta_k$ for   $\delta_k \to 0$ and let $\al_k = \al(\delta_k)$ with $\lim_k \al_k = \lim_k \delta_k^2 / \al_k = 0$. Any sequence $(\signal_k)_{k\in \N}$ with  $\To_{\al_k} (\signal_k,  \data_k) =0$  has at least one weak cluster point. Any such cluster point $\signal_\plus$ is a solution of \eqref{eq:limiting}, that is $\Ao \signal_\plus = \data$ and $\Ro(\signal_\plus) \in \ker(\Ao)^\perp$. If the solution of \eqref{eq:limiting} is unique, then $(\signal_k)_{k\in \N}$ weakly  converges to $\signal_\plus$.
\end{theorem}

\begin{proof}
Let $\signal^*$ be any solution of  $\Ao \signal = \data$. By definition of $\signal_k$ and the convexity of the data-fidelity term we have
\begin{align*}
    0 &= \inner{\Ao^* (\Ao \signal_k - \data_k) + \al_k \Ro(\signal_k), \signal^* - \signal_k}\\
    &\leq \frac{1}{2} \norm{\Ao \signal^* - \data_k}^2 - \frac{1}{2}  \norm{\Ao \signal_k - \data_k}^2  + \al_k \inner{\Ro(\signal_k), \signal^* - \signal_k}\\
    &\leq \delta_k^2/2 + \al_k \inner{\Ro(\signal_k), \signal^* - \signal_k} \,.
\end{align*}
Hence $2 \inner{\Ro(\signal_k), \signal_k - \signal^*} \leq \delta_k^2 / \al_k$. The choice of $\al_k$ and Condition~\ref{cond:convergence} show  that $(\signal_k)_k$ is bounded and hence has a weakly convergent subsequence.

Let  $\signal_\plus$ be the limit of  any weakly convergent subsequence denoted again by $(\signal_k)_k$. By the weak continuity of $\Ro$ we have that $(\Ro(\signal_k))_k$ is bounded and thus  $0 = \lim_k \To_{\al_k}(\signal_k, \data_k) = \lim_k \Ao^*(\Ao \signal_k - \data_k) + \alpha_k \Ro (\signal_k)  = \Ao^*(\Ao \signal_\plus - \data) $. Because $\data \in \ran(\Ao)$ this shows that $\signal_\plus$ is a solution of $\Ao \signal = \data$.  Moreover, for any $z_0 \in \ker(\Ao)$ we have
\begin{equation*}
    \inner{-\Ro(\signal_\plus), z_0} = \lim_k  \inner{-\Ro(\signal_k), z_0} = \lim_k  \inner{\Ao^*(\Ao \signal_k - \data_k), z_0}/ \al_k  = 0
\end{equation*}
which gives $\Ro(\signal_\plus) \in \ker(\Ao)^\perp$.
Finally,  if the solution of \eqref{eq:limiting}  is unique, then any subsequence of $(\signal_k)_k$ has subsequence weakly converging to $\signal_\plus$, which implies that the full sequence  weakly converges to $\signal_\plus$.
\end{proof}

Theorems \ref{thm:stability} and \ref{thm:convergence} show that solutions to the equilibrium equation~\eqref{eq:INV} are indeed stable and convergent whenever $\Ro$ satisfies Condition~\ref{cond:convergence}.

\subsection{Convergence rates}

The next goal is to derive quantitative results in the form of convergence rates. To this end, we introduce a novel concept called symmetric Bregman pairing, which we define below.

\begin{definition}[Symmetric Bregman pairing]
Let \(\Ro \colon X \to X\) be an arbitrary (not necessarily monotone)  operator. We define the symmetric Bregman pairing as
\begin{equation} \label{eq:abregman}
     \forall \signal, z \in \X \colon \quad
     d_\Ro(\signal, z) := \abs{\inner{\Ro(\signal) - \Ro(z),\, \signal - z}} \,.
\end{equation}
\end{definition}

The symmetric Bregman pairing is typically non-symmetric and fails to satisfy the triangle inequality and thus is not a metric in the strict mathematical sense. It nevertheless serves as a meaningful quantitative measure between the elements \(\signal\) and \(z\) useful for our analysis.

\begin{remark}[Relation to symmetrized Bregman distance]
If \(\Ro\) is monotone, then \(d_\Ro(\signal, z) = \inner{\Ro(\signal) - \Ro(z),\, \signal - z}\) does not require the absolute value. In particular, if \(\Ro = \nabla \reg\) is the gradient of a convex functional, then \(d_\Ro(\signal, z)\) reduces to the classical symmetrized Bregman distance used for convergence rates in variational regularization~\cite{burger2007error,resmerita2006error}. The symmetric Bregman pairing extends the symmetrized Bregman distance to non-variational and non-monotone settings.
\end{remark}

\begin{remark}[Possible interpretation] \label{re:bregmaninterpretation}
     Assume that $\Ro$ is a linear, bounded, not necessarily self-adjoint positive semidefinite operator (meaning $\inner{\Ro \signal, \signal} \geq 0$ for all $x \in \X$.) Then,  \( 0 \leq \inner{\Ro \signal, \signal} =  \inner{\Ro_{\textnormal{s}} \signal, \signal} \)
     where $ \Ro_{\textnormal{s}} = (\Ro + \Ro^*)/2$ is the symmetric part of $\Ro$. Since $\Ro_{\textnormal{s}}$ is self-adjoint and positive semidefinite, there exists an operator $\Qo$ such that $\Qo^* \Qo = \Ro_{\textnormal{s}}$, and therefore \(
     \inner{\Ro \signal, \signal} =  \inner{\Qo \signal, \Qo \signal}/2 =:  \norm{\signal}_\Qo^2 / 2 \). Thus, the symmetrized Bregman distance of a linear operator $\Ro$ is the  half of the square of a weighted norm $\norm{\cdot}_\Qo$. In the case where $\Ro$ is nonlinear but smooth, a similar interpretation holds locally around a given point. For small \( h \), and locally around \( z \in \X \), it holds that
\(
     \bdist{z + h}{z} \simeq \inner{D\Ro[z] h, h} =  \norm{h}_{\Qo[z]}^2 / 2 \)
where the weight \( \Qo[z] \) now depends on \( z \) and on the local behavior of \( \Ro \).
\end{remark}

As in classical variational regularization, obtaining convergence rates requires additional assumptions on the interplay between the operator and the limiting solution $\signal_\plus$. More precisely, convergence rates for INV regularization are derived under the following assumptions:

\begin{assumption}[Conditions for convergence rates] \label{cond:rate} \mbox{}\\
The element $\signal_\plus \in \X$ satisfies:
\begin{align} \label{eq:source} 
&\Ro(\signal_\plus) \in \ran(\Ao^*) 
\\
\label{eq:range}
&\exists c,\varepsilon > 0 \; \forall z \in \M_\varepsilon(\signal_\plus) \colon \quad 
\inner{\Ro(z), \signal_\plus - z} \leq c \norm{\Ao z - \Ao \signal_\plus}
\end{align}
where $\M_\varepsilon(\signal_\plus) \coloneqq \set{z \in \X \colon \Ro(z) \in \ran(\Ao^*) \wedge  \inner{\Ro(z), z - \signal_\plus} < \varepsilon }$.
\end{assumption}

\begin{remark}[The source condition \eqref{eq:source}]
We refer to~\eqref{eq:source} as source condition, in accordance with the gradient case $\Ro = \nabla \reg$, where this is an established concept in variational regularization \cite{scherzer2009variational}. It requires that $\Ro(\signal_\plus)$ is of the form $\Ro(\signal_\plus) = \Ao^* \omega$, with $\omega \in \Y$ a source element. Note that, according to Theorem~\ref{thm:convergence}, the limiting solutions of the \eqref{eq:INV} satisfy $\Ro(\signal_\plus) \in \ker(\Ao)^\perp$. In the ill-posed case, $\ran(\Ao^*) \subsetneq \ker(\Ao)^\perp$ is a non-closed, dense subset. Thus, the source condition strengthens the limiting condition and can be seen as an abstract smoothness condition for $\Ro(\signal_\plus)$.
     
\end{remark}

\begin{remark}[Range-based monotonicity control \eqref{eq:range}] \label{rem:control}
We refer to~\eqref{eq:range} as the range-based monotonicity control condition.
In fact, suppose~\eqref{eq:source} is satisfied. Then
\begin{align*}
\inner{\Ro(z), \signal_\plus - z} 
&= \inner{\Ro(\signal_\plus), \signal_\plus - z}
- \bdist{\signal_\plus}{z}
\\        
&= \inner{\Ao^* w, \signal_\plus - z}
- \bdist{\signal_\plus}{z}
\\
&\leq c \norm{\Ao (\signal_\plus - z)} 
- \bdist{\signal_\plus}{z}.
\end{align*}
Thus Condition~\eqref{eq:range} holds whenever $\Ro$ is  monotone. In particular, any gradient field $\Ro = \nabla \reg$ is a monotone operator, and therefore in this case Assumption~\ref{cond:rate} reduces to the classical source condition $\nabla \reg(\signal_\plus) \in \ran(\Ao^*)$. However, it can also hold for non-monotone operators, in which case it quantifies and controls the local deviation from monotonicity around $\signal_\plus$, measured in terms of the residual norm $\|\Ao z - \Ao \signal_\plus\|$. This is particularly relevant in inverse problems, where the image-space norm typically reflects a stronger topology, making the condition a key tool for balancing nonlinearity and stability when deriving convergence rates.
\end{remark}

\begin{remark}[Convergence rates in the stable case]
In the stable case where $\ran(\Ao)$ is closed, 
we have $\ran(\Ao^*) = \ker(\Ao)^\perp$, and the source condition reduces to the characterization $\Ro \in \ker(\Ao)^\perp$ from Theorem~\ref{thm:convergence}.  
Moreover, by definition, for any $z \in \M_\varepsilon(\signal_\plus)$, we have $\Ro(z) = \Ao^*(\eta)$ for some bounded $\eta \in \Y$. Thus,  
\[
\inner{\Ro(z), \signal_\plus - z}
=
\inner{\eta, \Ao \signal_\plus - \Ao z}
\leq
\norm{\eta} \norm{\Ao \signal_\plus - \Ao z}
\leq
c \norm{\Ao \signal_\plus - \Ao z}.
\]
As a consequence, in the stable case, the range-based monotonicity control condition~\eqref{eq:range} holds automatically. 
Assumption~\ref{cond:rate} is therefore satisfied without any additional assumptions on the limiting solution $\signal_\plus$, 
since Theorem~\ref{thm:convergence} guarantees that $\Ro(\signal_\plus) \in \ker(\Ao)^\perp$.  
In particular, $\ran(\Ao)$ is closed whenever $\X$ is finite-dimensional.
\end{remark}

In the following we write $\al \asymp \delta$  for  $\al=\al(\delta)$  if there exist constants $C_1, C_2 > 0$ such that $C_1 \delta \leq \al \leq C_2 \delta$ as $\delta \to 0$. 

\begin{theorem}[Convergence rates] \label{thm:rates}
Let $\data \in \ran(\Ao)$ and $(\data_k)_{k\in \N} \in \Y^\N$ be a sequence of data  satisfying $\norm{\data_k - \data} \leq \delta_k$ with $\delta_k \to 0$ and $\al_k \asymp \delta_k$.
Let $(\signal_k)_k$ satisfy $\To_{\al_k}(\signal_k,  \data_k) = 0$ and  denote by  $\signal_\plus$ the weak limit of $(\signal_k)_k$, possibly after restriction to a subsequence (see Theorem \ref{thm:convergence}) and assume   $\signal_\plus$ satsfies \eqref{eq:range}. Then, the source condition \eqref{eq:source} is sufficient and necessary for the convergence rates  
\begin{enumerate}[label = (\alph*)]
    \item \label{rate1} $\norm{\Ao \signal_k - \data_k} = \mathcal{O}(\delta_k)$ for $k \to \infty$
    \item \label{rate2}  $d_\Ro(\signal_k, \signal_\plus) = \mathcal{O}(\delta_k)$ for $k \to \infty$.
\end{enumerate}
\end{theorem}

\begin{proof}
We adapt the proof presented in \cite{obmann2023convergence}. Let us first assume that the source condition $\Ro(\signal_\plus) \in \ran(\Ao^*)$ holds. From the proof of Theorem~\ref{thm:convergence} we see   $ 2 \al_k \inner{\Ro(\signal_k), \signal_k - \signal_\plus} \leq \delta_k^2$.
By assumption on $\al_k$ we have  $\limsup_k \inner{\Ro(\signal_k), \signal_k - \signal_\plus} \leq 0$ and hence $\signal_k \in \M_\varepsilon(\signal_\plus)$ for $k$ sufficiently large. For the rest of the proof suppose $\signal_k \in \M_\varepsilon(\signal_\plus)$. By definition we have $    d_\Ro(z, \signal_\plus) = \bdist{z}{\signal_\plus} + \eta \inner{\Ro(\signal_\plus) - \Ro(z), z - \signal_\plus}$, where $\eta = 0$ if $\bdist{z}{\signal_\plus} \geq 0$ and $\eta = 2$ otherwise. By assumption, $\Ro(\signal_\plus) \in \ran(\Ao^*)$ and $\inner{\Ro(z), \signal_\plus - z} \leq c \norm{\Ao z - \Ao \signal_\plus}$ for any $z \in  \M_\varepsilon(\signal_\plus)$. Hence, $
    \inner{\Ro(\signal_\plus) - \Ro(z), z - \signal_\plus} \leq c \norm{\Ao (z - \signal_\plus)}$
for some $c \geq 0$ and   $ d_\Ro(z, \signal_\plus) \leq \bdist{z}{\signal_\plus} + c \norm{\Ao (z - \signal_\plus)}$.
In particular, for $z = \signal_k$, 
\begin{equation}\label{eq:rate-a2}
  d_\Ro(\signal_k, \signal_\plus) \leq \bdist{\signal_k}{\signal_\plus} + c \norm{\Ao (\signal_k - \signal_\plus)} \,.
  \end{equation}
By construction of $\signal_k$, the convexity of the data-fidelity term, the equality $\Ao \signal_\plus = \data$ and the assumption $\norm{\data - \data_k} \leq \delta_k$ we have
\begin{multline} \label{eq:rate-a3}
    \frac{1}{2} \norm{\Ao \signal_k - \data_k}^2 + \alpha_k \inner{\Ro(\signal_k), \signal_k - \signal_\plus}
    \\
    = \frac{1}{2} \norm{\Ao \signal_k - \data_k}^2 + \inner{\Ao^* (\Ao \signal_k - \data_k), \signal_\plus - \signal_k} \leq \frac{1}{2} \norm{\Ao \signal_\plus - \data_k}^2 \leq \frac{1}{2} \delta_k^2 \,.
\end{multline}
By the source condition $\Ro(\signal_\plus) \in \ran(\Ao^*)$ it holds
\begin{multline}\label{eq:rate-a4}
    \inner{-\Ro(\signal_\plus), \signal_k - \signal_\plus}
    = \inner{\Ao^* w, \signal_\plus - \signal_k}
    \\ \leq C \norm{\Ao \signal_k - \Ao \signal_\plus} \leq C \left( \delta_k + \norm{\Ao \signal_k - \data_k} \right).
\end{multline}
Finally, from \eqref{eq:rate-a2}-\eqref{eq:rate-a4} and  Young's product inequality we obtain
\begin{align*}
    \frac{1}{2} \norm{\Ao \signal_k - \data_k}^2 + \alpha_k d_\Ro(\signal_k, \signal_\plus) &\leq \frac{1}{2} \delta_k^2 + C_1 \alpha_k \delta_k + C_2 \alpha_k \norm{\Ao \signal_k - \data_k} \\
    &\leq \frac{1}{2} \delta_k^2 + C_1 \alpha_k \delta_k + C_3 \alpha_k^2 + \frac{1}{4} \norm{\Ao \signal_k - \data_k}^2 \,,
\end{align*}
for some constants $C_1, C_2 >0$. The rates   \ref{rate1},  \ref{rate2}  then follow with $\alpha_k \asymp \delta_k$.

Assume now conversely that \ref{rate1},  \ref{rate2} hold and define $w_k \coloneqq (\Ao \signal_k - \data_k) / \al_k$. Then $(w_k)_{k \in \N} $ is  bounded and thus has a weakly convergent subsequence $(w_{k'})_{k' \in \N} $ with weak limit $w^\plus$. Along this subsequence we have $-\Ro(\signal_\plus) = \lim_k -\Ro(\signal_{k'}) = \lim_k \Ao^* w_{k'} = \Ao^* w^\plus$ and thus $\Ro(\signal_\plus) \in \ran(\Ao^*)$.
\end{proof}

For $\Ro(\signal_\plus) \notin \ran(\Ao^*)$, Theorem \ref{thm:rates} implies that $d_\Ro(\signal_k, \signal_\plus)$ cannot converge at rate $\delta_k$. However, this does not mean that no convergence rate in the absolut symmetrized Bregman distance holds. Instead, if rates hold, these rates have to be  slower than  $\delta_k$. For example, it is easy to construct examples where the convergence rate is exactly $\delta_k^a$ for some $a \in (0,1)$.

\subsection{Stability estimates}

We next derive stability estimates for the regularized problem \eqref{eq:INV} for the special class of monotone regularization operators.

\begin{theorem}[Stability estimates] \label{thm:stabilityestimates}
Assume $\Ro$ is  monotone, let $\al > 0$, $\data_1, \data_2 \in \Y$
and $\signal_1, \signal_2 \in \X$ with $\To_\al(\signal_1,  \data_1) = \To_\al(\signal_2,  \data_2) = 0$. Then
\begin{enumerate}[label = (\alph*)]
    \item \label{stab1} $\norm{\Ao (\signal_1 - \signal_2)} \leq \norm{\data_1 - \data_2}$,
    \item \label{stab2} 
   $ d_\Ro(\signal_1,\signal_2) \leq 1 / (2 \al) \norm{\data_1 - \data_2}^2$.
\end{enumerate}
\end{theorem}

\begin{proof}
By construction of  $\signal_1$, $\signal_2$ and Young's product inequality, 
\begin{align*}
   \al d_\Ro(\signal_1,\signal_2)
   &= \al \bdist{\signal_1}{\signal_2} 
   \\ &= -\inner{\Ao^*(\Ao \signal_1 - \data_1) - \Ao^*(\Ao \signal_2 - \data_2), \signal_1 - \signal_2}\\
    &= -\norm{\Ao(\signal_1 - \signal_2)}^2 + \inner{\data_1 - \data_2, \Ao(\signal_1 - \signal_2)} \\
    &\leq -\norm{\Ao(\signal_1 - \signal_2)}^2 + \norm{\data_1 - \data_2} \norm{\Ao(\signal_1 - \signal_2)} \\
    &\leq - \norm{\Ao(\signal_1 - \signal_2)}^2/2 
    +  \norm{\data_1 - \data_2}^2/2 \,.
\end{align*}
Thus $\norm{\Ao(\signal_1 - \signal_2)}^2/2 + \al d_\Ro(\signal_1,\signal_2) \leq  \norm{\data_1 - \data_2}^2/2$ which yields  \ref{stab1}, \ref{stab2}.
\end{proof}

\subsection{Examples}

We conclude this section with simple examples involving linear forward operators and linear regularizers. In these examples, we consider a bounded linear operator \( \Ao \colon \ell^2_0(\N) \to \ell^2_0(\N) \), which may have a non-closed range and a non-trivial kernel. Here, \( \ell^2_0(\N) \) denotes the subspace of square-integrable sequences \( x \in \ell^2(\N) \) satisfying \( x_0 = 0 \). We present a regularization operator \( \Ro \) corresponding to the standard variational regularization setting, as well as a related operator for which our theory generalizes classical variational regularization.

\begin{example}[Gradient regularizer]
The standard  choice for the regularization operators \( \Ro \colon \ell^2_0(\N) \to \ell^2_0(\N) \) is as the gradient of a convex functional  \( \reg \colon \ell^2_0(\N) \to \R \). A basic instance  is the negative discrete Laplacian, defined by
\(
\Ro x = - x_{n-1} + 2 x_n - x_{n+1},
\)
with the convention \( x_0 = 0 \). Specifically, \( \Ro \) is linear and bounded, and it is the gradient of the quadratic coercive regularizer \( \reg(x) = \frac{1}{2} \sum_{n \in \N} \abs{x_{n+1} - x_n}^2 \), thus satisfying Assumption~\ref{cond:convergence} for stability and convergence (see Example~\ref{eq:basic}).  Moreover the variational form ensures existence and uniqueness of solutions of \( \Ao^*(\Ao \signal - \data) + \alpha \Ro x = 0 \), which in this case equals the optimality condition of the Tikhonov functional \( \norm{\Ao \signal - \data}^2/2 + \alpha \reg(\signal) \). Moreover from the monotonicity of the gradient we get the stability estimates of  Theorem~\ref{thm:stabilityestimates}.
Moreover, as the gradient of a convex regularizer, \( \Ro \) is a monotone operator, and the convergence rates conditions reduce to the source condition  \( \Ro x \in \ran \Ao^* \), the typical smoothness criterion for convergence rates in variational regularization.
\end{example}

Next we provide an example extending the variational setting where the above conditions are satisfied. 

\begin{example}[Non-gradient regularizer]  
Consider the forward difference operator \( \Ro \colon \ell^2_0(\N) \to \ell^2_0(\N) \), defined by  
\((\Ro x)_n = x_n - x_{n+1} \), as regularization operator. The operator \(\Ro\) is linear and bounded but not symmetric, and thus not of gradient form. In particular, standard variational regularization theory does not apply in this case. However, all theory developed in this section applies, as we will argue next.

Decomposing \(\Ro = \Ro_{\textnormal{s}}+ \Ro_{\textnormal{a}}\) into the sum of its symmetric part \(\Ro_{\textnormal{s}} = (\Ro + \Ro^*)/2\) and antisymmetric part \(\Ro_{\textnormal{a}} = (\Ro - \Ro^*)/2\), we have  
\((\Ro_{\textnormal{s}}(x))_n = x_n - (x_{n-1} + x_{n+1})/2\) and \( (\Ro_{\textnormal{a}}(x))_n = (x_{n-1} - x_{n+1})/2\). In particular, \(\Ro_{\textnormal{s}}\) is symmetric and positive definite, and for all \(x, z\) we have  
\(
\inner{\Ro(x), x - z} \geq \lambda_{\textnormal{min}} \norm{x}^2 - \norm{\Ro} \norm{x} \norm{z},
\)
which tends to \(\infty\) as \(\norm{x} \to \infty\). Thus, \(\Ro\) satisfies the convergence and stability conditions. Further, because \(\inner{\Ro(x), x} = \inner{\Ro_{\textnormal{s}}(x), x} \geq 0\), the operator is monotone and the stability estimates apply. Finally, according to Remark~\ref{rem:control}, the convergence rates condition reduces to the source condition \(\Ro(\signal) \in \ran(\Ro_{\textnormal{a}})\). The source condition for \(\Ro\) is clearly satisfied for different elements than the source condition for the symmetric part, thus widening the scope of potential applications.
\end{example}

\section{Analysis  for contraction residual regularizers} 
\label{sec:contraction}

We now turn to a particular case where the regularization operator $\Ro$ belongs to a specific class of nonlinear regularization operators $\reg$, not necessarily of gradient form, which we will refer to as contraction residual regularizers, defined below. This class has been used in \cite{gilton2021deep} to show convergence of iteration \eqref{eq:grad}. In contrast, in this paper, we will use this  class  for the derivation of a complete regularization theory and for strengthening the results in the previous section.

\begin{definition}[Contraction residual regularizers]
We call  $\Ro \colon \X \to \X$ a contraction residual regularizer if it is of the form $\Ro = \id - \Co$  for some contractive mapping $\Co \colon \X \to \X$. That is,  there exists a constant \( L \in [0,1) \) such that 
\begin{equation} \label{eq:degrad}
    \forall \signal, z \in \X \colon \norm{(\Ro - \id)(\signal) - (\Ro - \id)(z)} \leq L \norm{\signal - z},
\end{equation}
where $\id$ is the identity operator. 
\end{definition}

We will see that the theory in Section~\ref{sec:general} on convergence, stability, and stability estimates is applicable in this case and convergence rates are equivalent to the source condition $\Ro(\signal_\plus) \in \ran(\Ao)$. However, for contraction residual regularizers, even stronger results are derived, including existence and uniqueness of solutions to the INV equation and the limiting problem. We will consider the case where \eqref{eq:INV} is solved exactly (Subsection~\ref{ssec:exact}) as well as the   case where \eqref{eq:INV} is solved approximately with a certain tolerance (Subsection~\ref{ssec:inexact}).

Note that in the context of a learned regularizer, the training procedure in~\cite{gilton2021deep} ensures contractivity of the learned component by adopting the strategy from~\cite{ryu2019plug}. Specifically, spectral normalization is applied to all layers, ensuring that each layer has a Lipschitz constant bounded above by one.

\begin{remark}[Simple examples] \label{eq:simple}
One easily constructs contraction residual regularizers which are not of gradient form. Consider, for example,  a single layer neural network function  of the form $\Ro(\signal)  =  \signal - \Lo_2 \sigma(\Lo_1 \signal + b)$. In order to be of gradient form this requires  $\Lo_2 = \Lo_1^*$ and specific choices for $\sigma$, see for example  \cite{riccio2022regularization}. On the other hand when   $\norm{\Lo_2} \cdot \norm{\Lo_2} < 1$ and  $\sigma$  is non-expansive (such as the ReLU) then then $\Ro$ is a contraction residual operator  that of non-gradient form  if  $\Lo_2 \neq  \Lo_1^*$.  
\end{remark}

We start our analysis by basic  properties of contraction residual regularizers.

\begin{lemma}[Contraction residual regularizers] \label{lemma:equivalence}
    Let $\Ro = \id - \Co$ satisfy \eqref{eq:degrad} with $L < 1$. Then the following hold:
    \begin{enumerate}[label=(\alph*)]
    \item  Lipschtz-coontinuity:  \label{cont1}
    $\forall \signal, z \in \X \colon \norm{\Ro(\signal) - \Ro(z)} \leq (1+L) \norm{\signal - z}$

   \item Cocoervivity: \label{cont2} 
      $\forall \signal, z \in \X \colon  \inner{\Ro(\signal) - \Ro(z), \signal - z} \leq (1+L) \norm{\signal - z}^2$

    \item  Strong monotonicity: \label{cont3}
    $ \forall \signal, z \in \X \colon\inner{\Ro(\signal) - \Ro(z), \signal - z} \geq (1-L) \norm{\signal - z}^2$

    \item Inverse Lipschtz-continuity:   \label{cont4}
    $\forall \signal, z \in \X \colon \norm{\Ro(\signal) - \Ro(z)} \geq (1-L) \norm{\signal - z}$
    \item  \label{cont5} The operator  $\Ro \colon \X \to \X $ is one-to-one.
\end{enumerate}
\end{lemma}

\begin{proof}
Let $\signal, z \in \X$. From \eqref{eq:degrad}, we directly get \ref{cont1}, and with the Cauchy-Schwarz inequality, this yields \ref{cont2}. By \eqref{eq:degrad} and the Cauchy-Schwarz inequality, we get  
\begin{multline*}
    \inner{\Ro(\signal) - \Ro(z), \signal - z} = \inner{(\Ro - \id)(\signal) - (\Ro - \id)(z), \signal - z} \\
    + \inner{\signal - z, \signal - z} 
    \geq -L\norm{\signal - z}^2 + \norm{\signal - z}^2,
\end{multline*}
which is \ref{cont3}, and another application of the Cauchy-Schwarz inequality gives \ref{cont4}. Finally, \ref{cont4} implies \ref{cont5} .
\end{proof}

Note that for strictly contractive residual regularizers, one can even show that $\Ro$ is onto \cite[Chapter 22]{bauschke2017convex}, which, together with \ref{cont4}, implies that it has a Lipschitz-continuous inverse $\Ro^{-1} \colon \X \to \X$.

\subsection{Theory for exact  solution}
\label{ssec:exact}

We start with the case where $\To_\al(\signal, \data^\delta) = 0$, defined by \eqref{eq:INV}, is solved exactly. Lemma~\ref{lemma:equivalence} shows that $\Ro$ is strongly monotone and suggests that \eqref{eq:INV} behaves similarly to a variational regularization with a strongly convex regularizer. Indeed, the next theorem shows that we get similar results with a full convergence theory.

\begin{theorem}[Regularization with contraction residual regularizers] \label{thm:reg} 
Let $\Ro$ be a weakly  continuous contraction residual operator with $L < 1$. Then, the following hold:

    \begin{enumerate}[label = (\alph*)]
    \item \label{thm:reg1} \textbf{Existence:}
    $\forall \data^\delta \in \Y  \; \forall\al > 0 \colon \To_\al(\signal, \data^\delta) = 0$ has a unique solution.

    \item \label{thm:reg2} \textbf{Stability:} Let $\al >0$, $\data^\delta \in \Y$, $(\data_k)_{k\in \N} \in \Y^\N$  norm-converge to $\data^\delta$  and for any $k \in \N$ let $\signal_k \in \X$ solve $\To_\al(\signal_k, \data_k) = 0$. Then, the sequence $(\signal_k)_{k\in \N}$ norm-converges to the unique solution of $\To_\al(\signal, \data^\delta) = 0$.

    \item \textbf{Stability estimates:}
    \label{thm:reg3} Let $\data_1, \data_2 \in \Y$, $\al > 0$ and  $\signal_1, \signal_2$ satisfy $\To_\al(\signal_1, \data_1) = \To_\al(\signal_2, \data_2)= 0$. Then 
    \begin{enumerate}
        \item $\norm{\Ao (\signal_1 - \signal_2)} \leq \norm{\data_1 - \data_2}$,
        \item $d_\Ro(\signal_1,\signal_2) \leq 1 / (2 \al) \norm{\data_1 - \data_2}^2$,
        \item $\norm{\signal_1 - \signal_2} \leq \sqrt{1 / (2 \al (1-L))} \norm{\data_1 - \data_2}$.
    \end{enumerate}

    \item \textbf{Convergence:}
    \label{thm:reg4}
    Let $\data \in \ran(\Ao)$, $(\delta_k)_{k\in \N}, (\al_k)_{k\in \N} \in (0, \infty)^\N$, $(\data_k)_{k\in \N} \in \Y^\N$ with $\norm{\data - \data_k} \leq \delta_k$ and $\lim_k \al_k = \lim_k \delta_k^2 / \al_k = 0$ and  let $\signal_k$ be the unique solution of $\To_{\al_k}(\signal, \data_k) = 0$ for $k \in \N$. Then, $(\signal_k)_k$ norm-converges to the unique solution $\signal_\plus$ of $\Ao \signal = \data$ with $\Ro(\signal) \in \ker(\Ao)^\perp$.

        \item \textbf{Convergence rates:} \label{thm:reg5} 
    In the  setting of \ref{thm:reg4} with  $\al_k \asymp \delta_k$  for $k \to \infty$, the source condition  $\Ro(\signal_\plus) \in \ran(\Ao^*)$ is necessary and sufficient for the rates 
    \begin{enumerate}
        \item $\norm{\Ao \signal_k - \data_k} = \mathcal{O}(\delta_k),$
        \item $d_\Ro(\signal_k, \signal_\plus) = \mathcal{O}(\delta_k),$
        \item $\norm{\signal_k - \signal_\plus} = \mathcal{O}(\sqrt{\delta_k})$.
    \end{enumerate}

\end{enumerate}
\end{theorem}

\begin{proof}
Fix $\al, \beta > 0$, $\data^\delta \in \Y$ and define  $\Phi_{\al, \beta}(\signal) \coloneqq  \signal - \beta \left( \Ao^* (\Ao \signal - \data^\delta) + \al \Ro(\signal) \right)$. For  $\signal, z \in \X$ we have 
\begin{align*}
    &\norm{\Phi_{\al, \beta}(\signal) - \Phi_{\al, \beta}(z)} 
    \\
    &\qquad \leq 
    \norm{\left((1-\beta \al) \id - \beta \Ao^* \Ao \right) (\signal - z)} + 
    \al \beta \norm{(\Ro - \id)(\signal) - (\Ro - \id)(z)} 
    \\
    &\qquad \leq 
    \norm{(1-\beta \al) \id - \beta \Ao^* \Ao} \norm{\signal - z} + \al \beta L \norm{\signal - z} \,.
\end{align*}
For $\beta \leq 1 / (\snorm{\Ao}^2 + \al )$ we have $\norm{(1-\beta \al) \id - \beta \Ao^* \Ao} \leq 1 - \beta \al$ and hence  $\Ro$ is Lipschitz-continuous with Lipschitz constant $\gamma \leq (1 - \beta \al) + L \al \beta = 1 - \al \beta (1 - L) < 1$. Existence  and uniqueness of a fixed-point of $\Phi_{\al, \beta}$  thus follows  from Banach's fixed point theorem which gives  \ref{thm:reg1}.   Moreover, Lemma~ \ref{lemma:equivalence} gives the monotonicity  and thus \ref{thm:reg3},  \ref{thm:reg5} hold true. Item \ref{thm:reg2} is a consequence  of  \ref{thm:reg3}.

It remains  to verify \ref{thm:reg4}. Lemma~\ref{lemma:equivalence} shows $\inner{\Ro(\signal), \signal - z} \geq (1-L)\norm{\signal-z}^2 + \inner{\Ro(z), \signal - z}$ and thus Condition~\ref{cond:convergence} is satisfied.  Theorem~\ref{thm:convergence} gives the existence of a weakly convergent subsequence. We next show that the solution of the limiting problem \eqref{eq:limiting} is unique. To that end let  $\signal$ satisfy   $\Ao \signal = \data$ and $\Ro(\signal) \in \ker(\Ao)^\perp$.  Then $\signal_\plus - \signal \in \ker(\Ao)$, $\Ro(\signal_\plus) - \Ro(\signal) \in \ker(\Ao)^\perp$ and thus $\inner{\Ro(\signal_\plus) - \Ro(\signal), \signal_\plus - \signal} = 0$. By Lemma~\ref{lemma:equivalence}, $0 = \inner{\Ro(\signal_\plus) - \Ro(\signal), \signal_\plus - \signal} \geq (1-L) \norm{\signal_\plus - \signal}^2$ and thus $\signal = \signal_\plus$. Following the proof of Theorem~\ref{thm:convergence} we find $2 \al_k \inner{\Ro(\signal_k), \signal_k - \signal_\plus} \leq \delta_k^2$. By Lemma~\ref{lemma:equivalence}, $
        (1-L) \norm{\signal_k - \signal_\plus}^2 \leq \inner{\Ro(\signal_k) - \Ro(\signal_\plus), \signal_k - \signal_\plus} \leq  \delta_k^2/(2 \al_k) + \inner{\Ro(\signal_\plus), \signal_\plus - \signal_k}$.
Because $(\signal_k)_k$ converges weakly to $\signal_\plus$ we get the norm convergence which completes the proof.
\end{proof}

The stability estimates in Theorem~\ref{thm:reg} show that for   fixed $\al > 0$, the regularized  reconstruction  operator $\Bo_\al$ defined  simplicity  by  $\To_\al(\Bo_\al(\data), \data) = 0$) is  injective and  Lipschitz-continuous. Further note that we  derived stability with respect to the norm as well as the symmetric Bregman pairing. Opposed to the norm estimate, the estimate in the  symmetric Bregman pairing is independent of the constant $L$. Hence, for $L$ close to $1$ the stability estimate in the norm has a large stability constant, whereas the stability estimate for the symmetric Bregman pairing does not depend on this factor.

    \begin{figure}[htb!]
    \centering
    \begin{tikzpicture}

\draw[->, ultra thick] (-3,0) --  (3,0); 
\draw[->, ultra thick]  (0,-3) -- (0,3);

\draw[-, red, ultra thick] (-1, -3) -- (1, 3) node[above] {$\text{ker}(\Ao)^\perp$};

\draw[-, blue, ultra thick, name path=solutions] (-3, -1/3) node[below] {$L(\Ao, y)$} -- (2, -2);

\draw [ultra thick, orange, name path=manifold] (2,2) to[out=240,in=115] (1,1) to[out=-180+115,in=100] (1,-3) node[below] {$G^{-1}(\textnormal{ker}(\Ao)^\perp)$};

\coordinate [name intersections={of=solutions and manifold, by=solution}];
\node at (solution) [circle,fill,inner sep=2pt, label=below right:$x_+$] {};

\end{tikzpicture}
    \caption{Interpretation of the limiting problem. According to \eqref{eq:limiting-neu} the limiting problem chooses the unique solution of $\Ao \signal = \data$  with $ \signal \in \Ro^{-1} (\ker(\Ao)^\perp)$, where  $\Ro^{-1}|_{\ker(\Ao)^\perp} \colon \ker(\Ao)^\perp  \to  \X$ can be seen as a parametrization of the  manifold of all desired solutions. }
    \label{fig:manifold}
    \end{figure}

\begin{remark}[Interpretation\label{re:limitproblem}  of the limiting problem] 
As noted above, any contractive residual operator $\Ro$ is bijective with a Lipschitz continuous inverse. Consequently, the limiting problem \eqref{eq:limiting} can be reformulated as
\begin{equation} \label{eq:limiting-neu}
 	\text{Find } \signal \text{ such that } \quad  \signal \in \Ro^{-1}(\ker(\Ao)^\perp) \cap \Ao^{-1}(\{ \data \}) \,.
\end{equation}
In particular, the restricted inverse $\Ro^{-1}|_{\ker(\Ao)^\perp} \colon \ker(\Ao)^\perp \to \X$ defines a parametrization of a manifold on which the limiting solutions must lie. Intersecting this manifold with the solution set $L(\Ao, \data) := \Ao^{-1}(\{ \data \})$ for exact data $\data$, yields the unique solution of the exact data problem determined by $\Ro$. Figure~\ref{fig:manifold} provides an illustration of this.
\end{remark}

\begin{example}[Implicit null space networks]
Consider a regularization operator the form $\Ro  = \id  - \Po_{\ker(\Ao)} \Co$ where $\Po_{\ker(\Ao)}$ is the projection on $\ker(\Ao)$ and $\Co \colon \X \to \X$ is Lipschitz.  Then~\eqref{eq:INV} is satisfied if and only if  $\Ao^*(\Ao \signal - \data^\delta) + \al \Po_{\ker(\Ao)^\perp} \signal = 0$ and $\Po_{\ker(\Ao)}(\signal - \Co(\signal)) = 0$.  With the decomposition $\signal = \signal_1 + \signal_0$ with $\signal_1 \in \ker(\Ao)^\perp$ and $\signal_0 \in \ker(\Ao)$ we find that $\signal_1 = (\Ao^* \Ao + \al \id)^{-1} \Ao^* \data^\delta$ is defined  by classical Tikhonov regularization (variational regularization \eqref{eq:tikhonov} with  the regularizer $\reg(\signal) = \norm{\signal}^2/2$). The component $\signal_0$ is then implicitly defined by the equation  $ 0 = \signal_0 - \Po_{\ker(\Ao)} \Co(\signal_1 + \signal_0)$. Adding $\signal_1$ to both sides, using the definition $\Ro  = \id  - \Po_{\ker(\Ao)} \Co$ and applying $\Ro^{-1}$ yields  $\signal_0 = \Ro^{-1}(\signal_1) - \signal_1$. Hence, the regularized solution operators defined by  \eqref{eq:INV} is seen to have the form
 \begin{equation} \label{eq:nullnet}
 	\Bo_\al 
	= (\id  - \Po_{\ker(\Ao)} \Co)^{-1} \circ ( (\Ao^* \Ao + \al \id)^{-1} \Ao^* )\,.
\end{equation} 
This  is an instance of a regularized null space network proposed in \cite{schwab2019deep} as it is the composition  of classical  regularization $(\Ao^* \Ao + \al \id)^{-1} \Ao^*$ based on the Moore-Penrose inverse followed by trainable  network $\Ro^{-1} = (\id  - \Po_{\ker(\Ao)} \Co)^{-1}$ that adds  elements of the null space of  $\Ao$. If  $\Ro$ has already  be trained, then the null space network  \eqref{eq:nullnet}  is defined implicitly by~\eqref{eq:INV}.  This  is alternative interpretation of the null-space networks proposed in       \cite{schwab2019deep} where instead  of \eqref{eq:nullnet} the explicit form $
\Bo_\al 
	= (\id  - \Po_{\ker(\Ao)} \No)  \circ ( (\Ao^* \Ao + \al \id)^{-1} \Ao^* )$ with a Lipschitz mapping $\No$ has been  proposed in the context of learned regularization. 
\end{example}

\subsection{Theory for approximate solution of \eqref{eq:INV}}
\label{ssec:inexact}

We next study the situation  where INV equation~\eqref{eq:INV} is  only solved up to a certain tolerance. For that we  assume   $\Ro$  to be   a contraction residual  regularizer and analyze the particular case that the source condition \eqref{eq:source} is satisfied.

\begin{lemma}[Error estimates with tolerance $\varepsilon_k$]  \label{lemma:tolerance} 
Let $\Ro$ be a weakly  continuous contraction residual operator with $L < 1$ and let $\signal_\plus \in \X$  satisfy the source condition  $\Ro(\signal_\plus) \in \ran(\Ao^*)$. Further let $(\data_k)_k \in \Y^\N$, $(\signal_k)_k \in \X^\N$ and $(\varepsilon_k)_k, (\delta_k)_k \in (0, \infty)^\N$ satisfy   $\norm{\Ao \signal_\plus - \data_k} \leq \delta_k$, $\norm{\To_{\al_k}(\signal_k,  \data_k)} \leq \varepsilon_k$   and $\delta_k \to 0$ as $k \to \infty$. Then for the  any parameter choice $\al_k \asymp \delta_k$ there are constants   $C_1, C_2 > 0$ such that 
\begin{equation} \label{eq:rateapprox}
        \norm{\signal_k - \signal_\plus} \leq C_1 \sqrt{\delta_k} + C_2 \frac{\varepsilon_k}{\delta_k (1-L)}  \,.
    \end{equation}
    
\end{lemma}

\begin{proof}
Let $\signal_k^*$ solve the equilibrium equation $\To_{\al_k}(\signal_k^*,  \data_k) = 0$ exactly. By  Lemma~\ref{lemma:equivalence} we have $   \norm{\signal_k - \signal_k^*} \leq \norm{\To_k(\signal_1) - \To_k(\signal_k^*)} / (\alpha_k (1-L)) \leq  \varepsilon_k/ (\alpha_k (1-L))$. With the triangle inequality and the convergence rates of  Theorem~\ref{thm:reg} this shows  \eqref{eq:rateapprox}.
\end{proof}

If \eqref{eq:INV} is solved up to a tolerance $\varepsilon > 0$ independent of the noise level $\delta$, then Lemma~\ref{lemma:tolerance} gives the bound
\(
\norm{\signal_k - \signal_\plus} \leq C_1 \sqrt{\delta_k} + C_2 \varepsilon/( (1-L) \delta_k)
\).
Hence, in this case, the rate \eqref{eq:rateapprox} does not imply convergence of the sequence $(\signal_k)_k$ to $\signal_\plus$. Instead, it indicates a form of semi-convergence behavior: initially, as the noise level $\delta_k$ decreases, right hand side decreases until a certain point, after which it eventually diverges to infinity as $\delta_k \to 0$. To deduce convergence and convergence rates, the tolerance must be chosen depending on the noise level. For example, we can derive the following result.

\begin{theorem}[INV regularization with tolerance $\varepsilon_k$] \label{thm:tol}
In the setting of Lemma~\ref{lemma:tolerance} assume additionally that $\varepsilon_k =   \delta_k \eta_k$ with  $\eta_k \sim \sqrt{\delta_k}$ as $k \to \infty$. 
\begin{enumerate}[label=(\alph*)]
\item \textbf{Convergence and convergence rates:}
\label{thm:tol-a}
As $k \to \infty$,  
\begin{enumerate}[label=(\alph*)]
        \item $\norm{\signal_k - \signal_\plus} = \mathcal{O}(\sqrt{\delta_k} )$
        \item $d_\Ro(\signal_\plus, \signal_k)
        = \mathcal{O}(\delta_k)$
        \item $\norm{\Ao \signal_k - \data_k} = \mathcal{O}(\delta_k)$.
\end{enumerate}

\item \textbf{Finite number of iterations:}
\label{thm:tol-b}
A near-solution of \eqref{eq:INV} with tolerance $\varepsilon_k$ can be constructed with $\mathcal{O} (\log(\delta_k) / \left( (L-1) \delta_k \right))$ iterations of the fixed point iteration to $\Phi_k (\signal)=  \signal - \beta (\Ao^*(\Ao \signal - \data_k)   + \al \Ro(\signal)  )$.
\end{enumerate}
\end{theorem}

\begin{proof}
The rates for $\norm{\signal_k - \signal_\plus}$ and for the  symmetrized Bregman distance follow from  Lemmas~ and~\ref{lemma:tolerance}. Further, with $\To_k  = \To_{\al_k}(\cdot,  \data_k)$ we have    $\inner{\To_k(\signal_k) - \To_k(\signal_k^*), \signal_k - \signal_k^*} \leq C \delta_k \eta_k^2$. By Lemma~\ref{lemma:equivalence},  $\norm{\Ao(\signal_k - \signal_k^*)}^2 \leq C \delta_k \eta_k^2$ and hence $\norm{\Ao \signal_k - \data_k} \leq \Tilde{C}(\delta_k + \eta_k \sqrt{\delta_k})$ as desired.  To show the last claim we assume without loss of generality that $\snorm{\Ao} = 1$. According to the proof of Theorem~\ref{thm:reg} the mapping $\Ro_k$ has  contraction constant $\gamma_k = \left(1 + \al_k L \right) / \left(1 + \al_k \right)$. By  Banach's fixed point theorem, the iterates $\signal_k^n \coloneqq \Ro_k (\signal_k^{n-1}) $  satisfy $\norm{\signal_k^* - \signal_k^n} \leq C \gamma_k^n$. It is thus sufficient to have $C \gamma_k^n \leq \sqrt{\eta_k} \delta_k $. Rearranging this inequality we find that we need on the order of $\log(\delta_k) / \log(\gamma_k)$ iterations. Since $\log(\gamma_k) = \log(1 + \al_k L) - \log(1 + \al_k) \asymp (L-1) \al_k \asymp (L-1) \delta_k$, we thus get  $\mathcal{O} (\log(\delta_k) / \left( (L-1) \delta_k \right))$ iterations.
\end{proof}

Note that, in the context of Theorem~\ref{thm:tol}~\ref{thm:tol-b}, stability of the solutions after a finite number of iterations follows from the Lipschitz continuity of each iterative update. When using a finite number of iterations, the regularized solution depends on the initial value, and hence uniqueness of a near-solution is lost.  Furthermore, a similar proof can be given for step sizes satisfying $\eta_k \geq \eta_* > 0$. In this case, the resulting solution may not coincide with $\signal_\plus$ as characterized in Theorem~\ref{thm:reg}, but instead will be close to it, with the degree of closeness depending on $\eta_*$. Additionally, the convergence rates in the data-fidelity term are only of order $\sqrt{\delta}$.

\subsection{$\Ro$-recoverability}
\label{ssec:practical}

In this section, we denote by $\V \subseteq \X$ a set of signals of interest, which may, for example, represent natural images. We further assume that $\Ro$ is a weakly continuous contractive residual regularizer.  The solution operator for \eqref{eq:INV} is denoted by $\Bo_\alpha \colon \Y \to \X$, defined by $\data^\delta \mapsto \Bo_\alpha(\data^\delta)$ as the solution of the equilibrium equation $\To_\alpha(\Bo_\alpha(\data^\delta), \data^\delta) = 0$. For $\signal \in \V$, we refer to $\Ao \signal$ as the noise-free data, and to any $\data^\delta \in \Y$ satisfying $\snorm{\data^\delta - \Ao \signal} \leq \delta$ as noisy data.  Furthermore, for $\data \in \ran(\Ao)$, we denote by $\signal_\plus = \Bo_0(\data)$ the unique solution of the limiting problem \eqref{eq:limiting} (see Theorem~\ref{thm:reg}).

According to Theorem~\ref{thm:reg}, as the noise level tends to zero, the regularized solutions $\Bo_\alpha(\data^\delta)$ of the equilibrium equation converge to the limiting solutions $\Bo_0(\Ao \signal)$. This justifies the following definition.

\begin{definition}[$\Ro$-recoverability]
Element  $\signal \in \X$ is called $\Ro$-recoverable if $ \Bo_0 \Ao \signal = \signal$. The set $\V \subseteq \X$ is called $\Ro$-recoverable if all its elements are $\Ro$-recoverable. 
\end{definition}

The following theorem gives a characterization of $\Ro$-recoverability and provides a necessary condition on the interplay of $\V$ and $\ker(\Ao)$ for this to be possible.

\begin{theorem}[$\Ro$-Recoverability\label{thm:rec}] Let $\Ro$ be a weakly continuous contraction residual regularizer with   $L<1$.  Then, for  any subset $\V \subseteq \X$  and any $\signal \in \X$ the following hold:
    \begin{enumerate}[label = (\alph*)]
        \item \label{rec1} $\V$ is $\Ro$-recoverable $\Leftrightarrow$ $ \Ro (\V) \subseteq    \ker(\Ao)^\perp $.

        \item \label{rec2} $ \Ro (\V) \subseteq   \ker(\Ao)^\perp  \Rightarrow   \sup_{\signal_1, \signal_2 \in \V} \snorm{\Po_{\ker(\Ao)}(\signal_1 - \signal_2)} / \snorm{\signal_1 - \signal_2} \leq L < 1$.

        \item \label{rec3}
        $\forall \signal \in \X \colon
        \norm{\Bo_0(\Ao \signal) - \signal} 
        \geq \snorm{\Po_{\ker(\Ao)} \Ro(\signal )} / (1+L)$.
            \end{enumerate}
\end{theorem}

\begin{proof}
Item \ref{rec1}   is an immediate consequence of Theorem~\ref{thm:reg}. Now, if $\Ro(\V) \subseteq \ker(\Ao)^\perp$  and  $\signal_1, \signal_2 \in \V$, then $L \norm{\signal_1 - \signal_2} \geq \norm{(\Ro - \id)(\signal_1) - (\Ro - \id)(\signal_2)} \geq \norm{\Po_{\ker(\Ao)}(\signal_1 - \signal_2)}$ which gives \ref{rec2}. Finally from the  Lipschitz continuity of $\Ro$ we get  $ (L+1) \norm{\Bo_0(\Ao \signal) - \signal}  \geq \norm{\Ro (\Bo_0(\Ao \signal) ) - \Ro(\signal)}  \geq \snorm{\Po_{\ker(\Ao)} \Ro(\signal )}$  because $\Ro(\Bo_0(\Ao \signal)) \in \ker(\Ao)^\perp$ which shows~\ref{rec3}.
\end{proof}

Note that condition~\ref{rec3} is relevant for the case where $\Ro(\signal) \notin \ker(\Ao)^\perp$, as it then provides a lower bound on the recovery error and quantifies the degree of non-recoverability.

\begin{remark}[A-priori lower bound for $L$]
Condition \ref{rec2} gives a lower bound $L^* = \sup_{\signal_1, \signal_2 \in \V} \snorm{\Po_{\ker(\Ao)}(\signal_1 - \signal_2)} / \snorm{\signal_1 - \signal_2}$ for the  the contraction constant $L$ of the residual part $\Ro - \id$. Notably, this constant solely depends on $\V$ and $\ker(\Ao)$ and thus can be estimated before constructing  the regularizers $\Ro$.   
\end{remark}

Next we answer the existence of a contractive residual regularizer for the special case  that $\V \subset \X$ is  a closed subspace. In this case we denote by $\Po_\V$ the orthogonal projection onto $\V$.  
Existence for the existence of an arbitrary contractive residual regularizer is interesting question hat we hope to address in future work.

\begin{proposition}[Linear case]
Let $\V \subset \X$ be  a closed linear subspace with $\snorm{\Po_{\ker(\Ao)}   \Po_{\V}} < 1$. Then the operator $\Ro = \id -  \Po_{\ker(\Ao)} \Po_\V$ is  contraction residual  regularizer  such that $\V$ is $\Ro$-recoverable.
\end{proposition}

\begin{proof}
If  $\V \subseteq \V$ where $\V$ is a closed  subspace of $\X$ with $\norm{\Po_{\ker(\Ao)}   \Po_{\V}} < 1$, then  
$\Ro = \id -  \Po_{\ker(\Ao)} \Po_\V$ clearly satisfies \eqref{eq:degrad} with $L = \Po_{\ker(\Ao)} \Po_\V$. Moreover  $\Ro(\V) \subseteq \ker(\Ao)^\perp$, which  gives the $\Ro$-recoverability. 
\end{proof}

    \begin{figure}[htb!]
        \centering
        \begin{tikzpicture}

\draw[->, ultra thick] (-1,0) --  (3,0); 
\draw[->, ultra thick]  (0,-1) -- (0,3);

\draw[-, blue, ultra thick] (-1, -1) -- (2, 2) node[above] {$\text{ker}(\Ao)$};

\filldraw[orange] (0,1) circle (2pt) node[anchor=west]{$x_2$};
\filldraw[orange] (1,0) circle (2pt) node[anchor=north]{$x_1$};


\end{tikzpicture}
        \caption{A simple example of a linear equation in $\R^2$ where the subspace condition $\snorm{\Po_{\ker(\Ao)} \Po_{\V}} < 1$ does not hold for a linear space $\V$, despite $x_1$ and $x_2$ being $\Ro$-recoverable for a non-linear $\Ro$.}
    \label{fig:linear-restrictive}
    \end{figure}

\begin{remark}[Linear versus nonlinear $\Ro$]
While the linear subspace condition $\snorm{\Po_{\ker(\Ao)} \Po_{\V}} < 1$ is sufficient for the existence of a regularization operator for $\V$, it is restrictive. Consider, for example, Figure~\ref{fig:linear-restrictive}, where we want to recover elements $(x_1, x_2)$ in the positive closed cone $K \subseteq \R^2$ generated by $(1,0)$ and $(0,1)$ from the linear equation $x_1 - x_1 = y$. Geometrically, we see that this is easily possible with a non-linear operator $\Ro$. However, the smallest subspace $\V$ containing $K$ is the whole space $\R^2$, and thus the condition $\snorm{\Po_{\ker(\Ao)} \Po_\V} < 1$ cannot be satisfied. This example highlights that in many cases non-linear regularizers are favorable to reflect the non-linear nature of the solution set.
\end{remark}

Finally we show that  INV regularization can never yield  exact  reconstruction in the non-limiting case $\al>0$.

\begin{proposition}[Lower bound] \label{prop:lower}
    For any $\signal \in \X$  and $\al > 0$ we have 
    \begin{equation*}
        \norm{\signal - \Bo_\alpha \Ao \signal} \geq \frac{\al}{\snorm{\Ao}^2 + \al (1+L)} \snorm{\Ro(\signal)} \,.
    \end{equation*}
\end{proposition}

\begin{proof}
    The mapping $\signal \mapsto \To_{\al}(\signal, \data)$ is Lipschitz with $\lip(\To_{\al}) \leq \snorm{\Ao}^2 + \al (1+L)$. Hence, $\left( \snorm{\Ao}^2 + \al (1+L) \right) \norm{\signal - \Bo_\alpha \Ao \signal} \geq \norm{\To_{\al}(\Bo_\alpha \Ao \signal, \Ao \signal) - \To_{\al}(\signal, \Ao \signal)} = \al \snorm{\Ro(\signal)}$.
\end{proof}

Proposition \ref{prop:lower}    implies that the  solution of~\eqref{eq:INV} with data $\data = \Ao \signal$  recovers  $\signal$ if $\Ro(\signal) = 0$. As $\Ro$  is one-to-one this is only the case for  one  elements.

\section{Discussion and outlook} \label{sec:conclusion}

We presented a convergence analysis for solving inverse problems with the equilibrium equation~\eqref{eq:INV} including the derivation of stability estimates and convergence rates. This analysis was strengthened  in  the case where the regularization operator $\Ro$ satisfies   \eqref{eq:degrad}. In particular we derived the  limiting  problem \eqref{eq:limiting} for $\alpha \to 0$ which particular lead to a new  loss function for training the regularization operator $\Ro$. We have further shown  for finite $\alpha$, the equilibrium equation has limited  performance on any  given signal-class $\V$.

The results in our paper raise several new questions for future research. One potential direction is to weaken the assumption~\eqref{eq:degrad}. Other directions include studying a loss of the form
\begin{equation}\label{eq:loss-LIM}
 	\mathcal{L}_{\textnormal{LIM}}(\theta) 
	    \coloneqq \sum_i \snorm{\Po_{\ker(\Ao)} \Ro_{\theta}(\signal_i)}^2 \,,
\end{equation}
which is based on the limiting problem and targets $\Ro_{\theta}$-recoverability (Theorem~\ref{thm:rec})  of a set of training images $(\signal_i)_i$. In particular, analyzing this approach, comparing it to, and integrating it with DEQ and PnP training strategies constitute interesting lines of future research. Another promising direction is to derive regularization methods different from~\eqref{eq:degrad}, motivated by the structure of the limiting problem~\eqref{eq:limiting}.

While inspired by DEQ, we emphasize that our theory is not concerned with learning a particular regularizer, but rather focuses on the regularization properties of such approaches. Our framework applies broadly and includes variational regularization, DEQ, PnP, and RED methods. A key goal of our analysis is to accommodate learned regularizers, which are typically not gradients. Since actual numerical performance is highly dependent on the architecture used for $\Ro$ and the specific training strategy, we have deliberately chosen not to include numerical results. Including them could misleadingly suggest that we advocate a particular learning approach, which is not the case. Any learned or non-learned implicit regularization method of the form~\eqref{eq:INV} falls within the scope of our theoretical framework.

\end{document}